\documentclass{article}[12pt]
\usepackage{amssymb}
\usepackage{latexsym}
\usepackage{graphicx}

\newtheorem{theorem}{Theorem}

\newtheorem{corollary}[theorem]{Corollary}

\newtheorem{lemma}[theorem]{Lemma}

\newtheorem{proposition}[theorem]{Proposition}

\newenvironment{proof}[1][Proof]{\noindent\textbf{#1.} }{\ \rule{0.5em}{0.5em}}
\def\lab(#1)#2{\put(#1){\makebox(0,0)[c]{#2}}}

\begin{document}

\title{A Flow-dependent Quadratic Steiner Tree Problem in the Euclidean Plane\thanks{This research was supported by an ARC
Discovery Grant.}}

\author{M.~Brazil \and C.~J.~Ras \and  D.~A.~Thomas}

\date{}
\maketitle

\begin{abstract}
We introduce a flow-dependent version of the quadratic Steiner tree problem in the plane. An instance of the problem on a set of embedded
sources and a sink asks for a directed tree $T$ spanning these nodes and a bounded number of Steiner points, such that $\displaystyle\sum_{e \in
E(T)}f(e)|e|^2$ is a minimum, where $f(e)$ is the flow on edge $e$. The edges are uncapacitated and the flows are determined additively, i.e.,
the flow on an edge leaving a node $u$ will be the sum of the flows on all edges entering $u$. Our motivation for studying this problem is its
utility as a model for relay augmentation of wireless sensor networks. In these scenarios one seeks to optimise power consumption -- which is
predominantly due to communication and, in free space, is proportional to the square of transmission distance -- in the network by introducing
additional relays. We prove several geometric and combinatorial results on the structure of optimal and locally optimal solution-trees (under
various strategies for bounding the number of Steiner points) and describe a geometric linear-time algorithm for constructing such trees with
known topologies.\\


\end{abstract}

\section{Introduction}
Given a set of points $Z$ in a normed plane $\langle \mathbb{R}^2,||\cdot||\rangle$ and a real number $p>0$, the \textit{geometric power-$p$
Steiner tree problem} (or geometric $p$-STP) seeks a finite set of points $S\subset \mathbb{R}^2$ (the \textit{Steiner points}) and a tree
$T=\langle V(T), E(T)\rangle=\langle Z\cup S,E(T)\rangle$ such that $\displaystyle\sum_{uv\in E(T)}||u-v||^p$ is a minimum. For $p>1$ the input
to $p$-STP must include a strategy for bounding the number of Steiner points, without which a minimum solution may not exist. When $p=1$ we
obtain the classical Steiner tree problem \cite{gilbert,hwang}, which has been extensively studied under the rectilinear and the Euclidean
norms. Soukop \cite{soukop} was the first to explore the notion of non-linear networks from a topological point of view; he realised its
importance as a model for the design of transportation or communication systems. The operations research community has studied a very similar
problem in the form of the \textit{non-linear multi-facility location problem}; see for instance \cite{hook}. The $p$-STP, in the form given
above with $||\cdot||$ the Euclidean or rectilinear norm, was introduced by Ganley and Salowe in \cite{ganley,ganley3} for its application to
VSLI routing algorithms.

The $p=2$ case, which we refer to as the \textit{quadratic} STP, is particularly important for transportation problems \cite{white} and for some
wireless network problems. In the latter case this is because most of the energy of the network is utilised during data transmission, and,
furthermore, energy consumption is proportional to the transmission distance raised to an exponent $\alpha\in [2,5]$; see for instance
\cite{karl}. Since $\alpha$, the so-called \textit{path loss exponent}, is greater than $1$, adding a relay between any two communicating nodes
will lead to a reduction in total energy consumed. This leads us to an effective method of reducing power consumption in wireless sensor
networks through \textit{relay augmentation}. In \textit{free-space} $\alpha$ can be shown to be exactly $2$ \cite{gold}; moreover, the
$\alpha=2$ case is important in some real-world sensor network applications when constructive interference applies, such as in beamforming and
communication through corridors \cite{karl}. We therefore solely address the quadratic STP in this paper, but with an additional flow component
that makes for a more realistic model of relay augmentation in wireless networks.

In order to make the above discussion more rigorous we note that the energy consumed by nodes in a wireless sensor network when a single packet
of data is transmitted over a distance $r$ in free space is $\beta r^\alpha+p_{\mathrm{rec}}=\beta r^2+p_{\mathrm{rec}}$, where $\beta$ is a
constant and $p_{\mathrm{rec}}$ is the energy required to receive a packet. Since $p_{\mathrm{rec}}$ is usually small relative to $r^2$ we
simplify and normalise the energy consumption to $\mathbb{E}(uv)=|uv|^2$, where $u$ is the transmitting node, $v$ is the receiving node,
$|\cdot|$ is Euclidean distance, and $uv$ denotes the edge $(u,v)$. In order to model power consumption we must include the rate of data flow
between $u$ and $v$, say $f(uv)$, to get $\mathbb{P}(uv)=f(uv)|uv|^2$. We assume that there is an additive flow function on the nodes, so that
the difference between the flow rate entering a given node and leaving the node equates to the supply rate at the node. It is therefore a simple
matter to calculate $f(uv)$ for any $u,v$ once we have the supply rates at all sensors (sources) and the topology of the network has been given.

The central problem of this paper is referred to as the \textit{flow-dependent quadratic Steiner tree problem} (FQSTP). An instance of the FQSTP
has an input of $n+1$ points in the plane -- i.e., $n$ sources and one sink -- and a supply rate at each source. We are also given a strategy
for bounding the number of Steiner points. The output is a set of Steiner points (satisfying the given bound) in the plane and a tree $T$
interconnecting all nodes such that the sum of $\mathbb{P}(u,v)$ over all $u,v$ is minimised. We can view the FQSTP as a weighted version of the
quadratic Steiner tree problem, where the weights are the flows on the edges. A related flow-dependent problem is the Gilbert arborescence
problem (GAP) \cite{gilbert2,volz}, which also has an additive flow function at the nodes but where the cost of an edge $uv$ is $w(f_{uv})|uv|$,
where $w$ is some increasing concave function. Similar to the FQSTP, the GAP has applications in telecommunications and transportation networks.
The computational complexity of the FQSTP is unknown; in fact, this is true for the complexity of the geometric power-p Steiner tree problem for
any $p$ besides $1$. Ganley \cite{ganley} observes that the methods used for proving NP-hardness of the classical geometric Steiner tree problem
fail for $p>1$ because of the lack of the triangle inequality; however, when the number of Steiner points is bounded by some constant strictly
less than $|Z|-2$ then the quadratic Steiner tree problem can be shown to be NP-hard by reduction from the \textit{geometric dominating set}
problem. Berman and Zelikovsky \cite{zeli} show that the graph version of the $p$-STP (where the Steiner points are restricted to being vertices
of a given graph) is MaxSNP-hard.

Our main contributions in this paper are an analysis of some of the important geometric properties of optimal (or locally optimal) solutions to
the FQSTP, and a new linear time geometric algorithm for the construction of locally minimal trees. This algorithm can be used as a component in
an exact solution to the FQSTP, and we also believe that it will lead to effective pruning techniques for such an exact solution (similar to
those used in the program GEOSTEINER for solving the Euclidean and Rectilinear STP \cite{warme}). We provide a formal definition of the FQSTP in
Section \ref{secNot}. In Section \ref{secBound} we describe various strategies for bounding the number of Steiner points in order to ensure that
a solution exists. Then, in Section \ref{secProp}, we state and prove a number of structural and geometric results on locally minimal solutions
to the FQSTP for each of the afore-mentioned bounding methods. Finally, Section \ref{secTop} presents our geometric linear-time algorithm for
constructing locally minimal trees of a given topology.

\section{Problem definition and notation}\label{secNot}
Let $Z\cup \{z_{\mathrm{BS}}\}$ be a set of $n$ \textit{sources} and a \textit{sink} (or \textit{base station}) $z_{\mathrm{BS}}$ embedded in
the plane; we assume that $Z\neq \emptyset$ and $z_{\mathrm{BS}}\notin Z$. We assume that there is a supply  $w(z_i)>0$ associated with each
member of $z_i\in Z$ and a demand $\displaystyle\sum_{z_i\in Z}w(z_i)$ associated with $z_{\mathrm{BS}}$. These supplies and demand determine
the flow on the resulting tree. Let $S$ be any finite set of additional points in the plane, referred to as the \textit{Steiner points}. A
directed tree $T$ spanning $Z\cup S\cup \{z_{\mathrm{BS}}\}$ is a \textit{flow-dependent quadratic Steiner tree} (FQST) if and only if

\begin{enumerate}
    \item Every edge $e$ of $T$ is labelled by a positive real number $f(e)$, its \textit{flow},
    \item $T$ is directed towards $z_{\mathrm{BS}}$ such that every node $z$ of $T$, except the sink, has exactly one outgoing edge $e_z^+$
    (\textit{out-edge}) with flow $f_z^+=f(e_z^+)$ and possibly some incoming edges (\textit{in-edges}) with sum of flows
    $f_z^-$. The sink has no out-edges, but has at least one in-edge,
    \item For the sink we have $f_{\mathrm{BS}}^+=0$ and $f_{\mathrm{BS}}^-=\displaystyle\sum_{z_i\in Z}w(z_i)$,
    \item For each source $z$ we have $f^+_{z}-f^-_z=w(z)$,
    \item For each Steiner point $s$, $f^+_s=f^-_s$,
    \item The \textit{cost} of $T$ is $L(T) :=\displaystyle\sum_{e_i\in E(T)} f(e_i)\vert e_i \vert^2$, where $E(T)$ is the edge-set of $T$.
\end{enumerate}

The objective of the FQSTP is to minimise $L(T)$ over all FQSTs. The decision variables for this problem are the number and locations of the
Steiner points and the topology of the tree interconnecting all points. An FQST minimising $L(T)$ will exist if and only if $\vert S\vert$ is
bounded, and below we discuss various strategies for doing so. An optimal solution will be referred to as a \textit{minimum flow-dependent
quadratic Steiner tree} (MFQST). Throughout this paper we assume that $w(z_i)=1$ for all $z_i\in Z$, but all our results can be generalised to
any positive $w(z_i)$.

It should be noted that points (3)-(5) in the definition of an FQST define an \textit{additive flow} function on the nodes. This means that the
flows from the sources are, in some sense, independent of each other; this fact leads to a sharing of some properties between MFQSTs and
shortest path trees. The wireless network analog of this is that data aggregation (for instance compression) does not take place at the nodes.

For any node $z$ the neighbours of $z$ incident to the in-edges of $z$ will be called $z$'s \textit{in-neighbours}, and we have a similar
definition for the \textit{out-neighbour} of $z$, which we sometimes refer to as the \emph{local sink} of $z$. The \textit{degree} of a node is
the number of edges incident to that node, and its \textit{in-degree} is the number of in-edges incident to it. We denote an edge or a line
segment connecting points $u$ and $v$ by $uv$, and its Euclidean length by $|uv|$. The familiar notation $||u||$ is used for the length of $u$
considered as a vector, i.e., $||u||=\sqrt{u_x^2+u_y^2}$ where $u=(u_x,u_y)$.

Any tree network $T$ interconnecting some or all of the nodes of $Z\cup S\cup\{z_{\mathrm{BS}}\}$ induces a \textit{tree topology}
$\mathcal{T}$, which is simply the labelled graph corresponding to the network $T$. If, in $\mathcal{T}$, every $z\in Z\cup \{z_{\mathrm{BS}}\}$
is of degree one and every $s\in S$ is of degree larger than one then $\mathcal{T}$ is a \textit{full topology}. Since Steiner points are never
of degree one in a MFQST, the edge set of any $\mathcal{T}$ induced by an MFQST on  $Z\cup S\cup \{z_{\mathrm{BS}}\}$ can be partitioned such
that every member of the partition induces a full topology. A tree is called \textit{degenerate} if and only if it has at least one edge of zero
length.

\section{Bounding the number of Steiner points}\label{secBound}
As stated before, an MFQST on a set of nodes will exist if and only if there is a bound on $|S|$. To see this suppose first that there is no
bound on $|S|$. We can reduce the cost of any FQST on the given nodes by adding one or more degree-two Steiner point to any edge; see Fig.
\ref{figBead}. Note that the total cost of such a  (straight-line) path is
$(p+1)\left(\displaystyle\frac{|uv|}{p+1}\right)^2f(uv)=\displaystyle\frac{|uv|^2f(uv)}{p+1}$ where $u,v$ are the end-points of the path, and
the $s_i$ are equally spaced degree-two Steiner points (in Fig. \ref{figBead}, and throughout, the sources and sink are shown as filled circles
and Steiner points are open circles). Since, if there is no bound on $|S|$, we can keep adding degree-two Steiner points, there is no optimal
solution; in fact, the cost of the tree will tend to zero as the number of Steiner points increases. On the other hand, if $|S|$ is bounded then
there are a finite number of possible tree topologies interconnecting the nodes, with each topology obtaining a unique minimum (as we show in
the next section). Therefore an optimal solution must exist.

\begin{figure}[htb]
\begin{center}
    \includegraphics[scale=0.55]{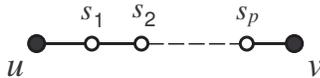}\\

  \caption{Total cost of path is $\frac{|uv|^2f(uv)}{p+1}$}
  \label{figBead}
\end{center}
\end{figure}

A popular method of bounding the number of Steiner points in the power-$p$ Steiner tree problem, $p>1$, is to restrict the degrees of Steiner
points to be greater than two \cite{zeli,ganley,ganley2,soukop}. It is clear that we then have an implicit upper bound of $k=|S|\leq n-1$.
Alternatively we could introduce a cost on the Steiner points. This latter bound is the one that is most relevant to the application motivating
this paper, namely wireless sensor network deployment \cite{wang2,xin,xu}. In this scenario the Steiner points correspond to transmitting
relays, which have an (often significant) associated cost. In practice, there may also be a need to impose edge-length bounds on the network,
but we do not study such bounds in this paper.

Let $T$ be an FQST. In summary, we consider the following three $|S|$-bounding strategies.

\begin{enumerate}
    \item The \textit{degree bound} stipulates a fixed value $\phi\geq 3$ such that $\phi\leq \deg s$ for any Steiner point $s$ in $T$.
    \item The \textit{explicit bound} stipulates a fixed upper-bound of $k$ for $|S|$. It should be clear then that $|S|=k$ since adding a Steiner
    point always leads to an improvement in cost.
    \item The \textit{node-weighted} version of the problem assigns a weight $c>0$ to every Steiner point. The objective of the
    \textit{node-weighted} FQSTP is then to minimise $L_c(T):=\displaystyle\sum_{e\in E(T)} f(e)\vert e \vert^2+c|S|$.
\end{enumerate}

\section{Properties of locally minimal FQSTs}\label{secProp}
Here we consider combinatorial and geometric structural properties of MFQSTs and locally minimal (with respect to a given topology) FQSTs. In
the first subsection we describe general properties that hold true under any $|S|$-bounding strategy. We then look at properties relating
specifically to degree-bounded FQSTs, and finally we look at properties for FQSTs that are not degree-bounded.

\subsection{General properties}
Suppose that we are given a set of points $Z \cup\{z_{\mathrm{BS}}\}$ embedded in $\mathbb{R}^{2}$, a set $S=\{s_1,...,s_k\}$ of (free) Steiner
points and a  topology $\mathcal{T}$ interconnecting $Z\cup S\cup\{z_{\mathrm{BS}}\}$. For any $\mathbf{u}=\langle u_1,...,u_k\rangle\in
\mathbb{R}^{2k}$ (where each $u_i \in \mathbb{R}^2$) let $T_{\mathbf{u}}$ be the tree with topology $\mathcal{T}$ that is obtained by embedding
Steiner point $s_i$ at location $u_i$ for every $i$. We wish to minimise $L(T_{\mathbf{u}})=\displaystyle\sum_{e\in E(T_{\mathbf{u}})}f(e)|e|^2$
over all $\mathbf{u}\in \mathbb{R}^{2k}$. Such a minimum is unique, when $T_{\mathbf{u}}$ is non-degenerate, by the strict convexity of
$L(T_{\mathbf{u}})$ (noting that it is a sum of strictly convex functions). A \textit{locally minimal tree} with respect to $\mathcal{T}$ is a
tree $T_\mathbf{u}$ of topology $\mathcal{T}$ minimising $L(T_\mathbf{u})$. Clearly any MFQST is locally minimal with respect to its topology.

Now let $Z=\{z_1,...,z_n\}$ be any set of points in the Euclidean plane, and suppose that we associate a mass $f_i$ with $z_i$ for each $i\in I
:= \{1, \ldots , n\}$. We use the familiar mass-point geometry notation $f_iz_i$ to refer to this weighted point, and we denote the
\textit{centre of mass} of the system of points $M=\{f_iz_i\:i\in I\}$ by $\mathcal{C}(M)$ or (by a slight misuse of notation)
$\mathcal{C}(f_1z_1,...,f_nz_n)$.

\begin{figure}[htb]
\begin{center}
    \includegraphics[scale=0.5]{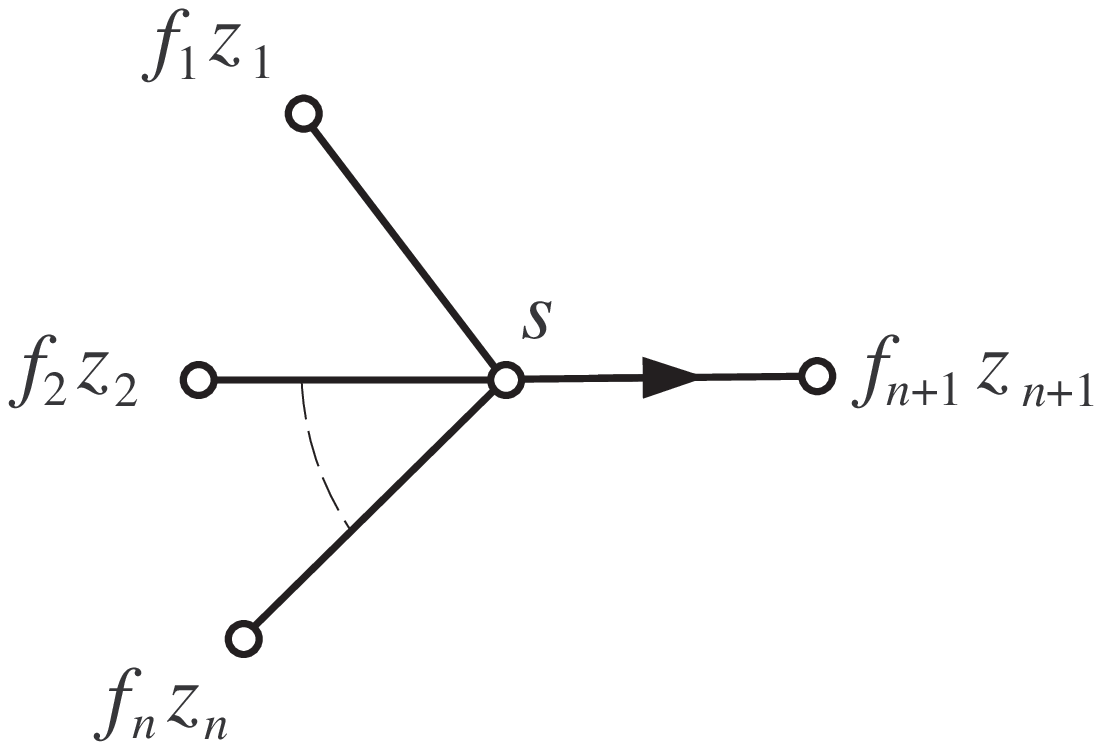}\\

  \caption{$s=\mathcal{C}(f_1z_1,...,f_{n+1}z_{n+1})$, where $f_{n+1}=\displaystyle\sum_{i\leq n} f_i$}
  \label{fig3}
  \end{center}
\end{figure}

\begin{proposition}\label{propCen}Suppose that a Steiner point $s$ in a locally minimal FQST $T$ has in-neighbours $z_1,...,z_n$ providing respective
flows $f_1,...,f_n$, and out-neighbour $z_{n+1}$ such that $f_{n+1}:=f(e_s^+)=\displaystyle\sum_{i\leq n} f_i$ (see Fig. \ref{fig3}). Then
$s=\mathcal{C}(f_1z_1,...,f_{n+1}z_{n+1})$.
\end{proposition}
\begin{proof}Let $I=\{1,...,n\}$. Suppose that we perturb the Steiner point $t$ units away from $s$ in the direction of the unit vector
$u$ and let the resultant tree be $T_0$. Then $$\psi_{u}(t):=L(T_0)= \displaystyle\sum_{i\in I}f_i\vert\vert
s+tu-z_i\vert\vert^2+f_{n+1}\vert\vert s+tu-z_{n+1}\vert\vert^2
$$ attains a minimum at $t=0$. But
$$\psi^\prime_{u}(t)=2\displaystyle\sum_{i\in I}f_i.\langle u,s+tu-z_i\rangle+
2f_{n+1}.\langle u,s+tu-z_{n+1}\rangle$$ and therefore
$$\psi^\prime_{u}(0)=2\displaystyle\sum_{i\in I}f_i.\langle u,s-z_i\rangle+
2f_{n+1}.\langle u,s-z_{n+1}\rangle=0.$$ Therefore $$\left\langle u,\displaystyle\sum_{i\in I}f_i(s-z_i)\right\rangle=\left\langle
u,f_{n+1}(z_{n+1}-s)\right\rangle.$$ Since this holds for any unit vector $u$ we get
$$\displaystyle\sum_{i\in I}f_i(s-z_i)=f_{n+1}(z_{n+1}-s).$$ Therefore
$$s=\displaystyle\frac{\displaystyle\sum_{i\in
I}f_iz_i+f_{n+1}z_{n+1}}{\displaystyle\sum_{i\in I}f_i+f_{n+1}}$$ and the result follows since this is an expression for the weighted mean.
\end{proof}

According to the method of mass-point geometry \cite{cox}, any point $\mathcal{C}(f_1z_1,...,f_nz_n)$ can be constructed geometrically by
recursively merging masses and subdividing line segments into appropriate ratios. Consider for instance the construction of
$\mathcal{C}(f_1z_1,f_2z_2,(f_1+f_2)z_3)$ where $f_1,f_2>0$ and the $z_i$ are any three points in the plane, as illustrated in
Fig.~\ref{figMerge}. We first merge $f_1z_1$ and $f_2z_2$ into the point $(f_1+f_2)z_{1,2}$ where $z_{1,2}$ is a point on the line segment
$z_1z_2$ such that $\displaystyle\frac{|z_1z_{1,2}|}{|z_2z_{1,2}|}=\frac{f_1}{f_2}$. Merging $(f_1+f_2)z_{1,2}$ and $(f_1+f_2)z_3$ yields the
point $s=\displaystyle\frac{z_{1,2}+z_3}{2}$ since $z_{1,2}$ and $z_3$ have equal masses. We extend this merging method in Section \ref{secTop}
in order to construct locally minimal FQSTs for more general topologies.

\begin{figure}[htb]
\begin{center}
    \includegraphics[scale=0.45]{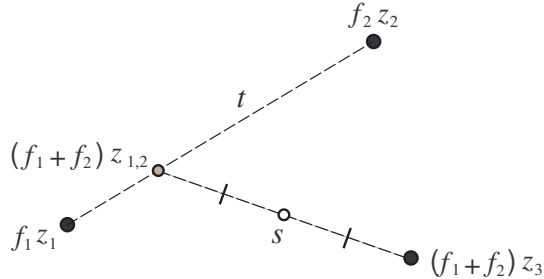}\\

  \caption{Merging masses to construct a Steiner point $s$. Here $t=\displaystyle\frac{f_1}{f_1+f_2}|z_1z_2|$.}
  \label{figMerge}
  \end{center}
\end{figure}

The reasoning behind the following corollary should now be obvious.

\begin{corollary}\label{corMid}Any Steiner point in a locally minimal FQST lies at the mid-point of its out-neighbour and the centre of mass of its
in-neighbours, where masses are assigned to the neighbours of the Steiner point as in Proposition \ref{propCen}.
\end{corollary}

Fig. \ref{figRel} shows an example of a locally minimal tree on four sources. Note the length ratios that result from the previous corollary; in
particular, $t=\mathcal{C}(1z,3s_2)$, and $s_1$ lies at the midpoint of $t$ and $z^\prime$.

\begin{figure}[htb]
\begin{center}
    \includegraphics[scale=0.5]{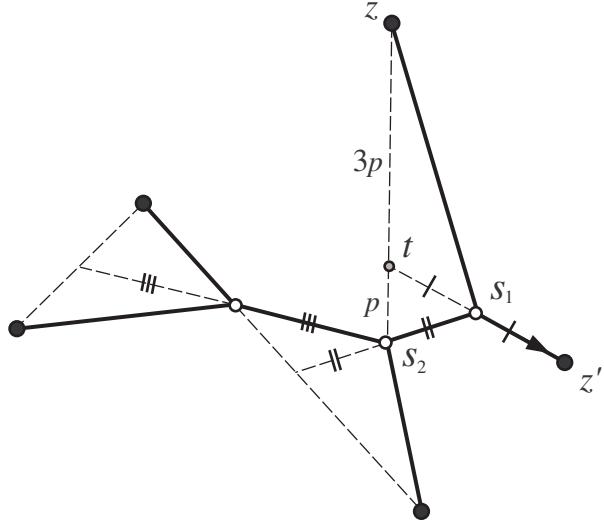}\\

  \caption{A locally minimal tree on four sources and a sink}
  \label{figRel}
  \end{center}
\end{figure}

Next we show that the converse of Proposition \ref{propCen} is also true. The proof is similar to Ganley's proof for quadratic Steiner trees
\cite{ganley}.

\begin{theorem}\label{thrIff}An FQST is locally minimal if and only if every Steiner point lies at the centre of mass of its neighbours.
\end{theorem}
\begin{proof}
We only need to prove sufficiency. Let $T$ be an FQST with Steiner points $s_1,...,s_p$, $p>0$. Let the in-neighbours of $s_i$ be
$x_i^1,...,x_i^{d_i}$, providing flows $f_i^1,...,f_i^{d_i}$, and let the out-neighbour of $s_i$ be $x_i^\prime$ which receives a flow of
$f(e^+_{s_i})=\displaystyle\sum_{j\leq d_i} f_i^j$. Furthermore, suppose that every Steiner point of $T$ is at the centre of mass of its
neighbours, so that
\[s_i=\displaystyle\frac{\displaystyle\sum f_i^jx_i^j+x_i^\prime\sum f_i^j}{2\displaystyle\sum f_i^j}\]
where all three sums are for $1\leq j\leq d_i$. Therefore

\begin{equation}\label{eq1}
2s_i\displaystyle\sum f_i^j-\displaystyle\sum f_i^jx_i^j-x_i^\prime\sum f_i^j=0.
\end{equation}

Let $A$ be the $p\times p$ square matrix containing the coefficients of the $s_i$ in the previous equation; note that some (or all) of the
$x_i^j$ or $x_i^\prime$ may be Steiner points. This creates a system of linear equations $A\mathbf{s}=\mathbf{b}$, where $\mathbf{s}=\langle
s_1,...,s_p\rangle$, and $\mathbf{b}$ is a constant derived from the fixed locations of the sources and the sink. By Proposition \ref{propCen}
one of the solutions to this system must produce a locally minimal tree. However, by Equation (\ref{eq1}) the diagonal entry in the $i$th row of
$A$ is $2\displaystyle\sum f_i^j$, which has magnitude strictly larger than any other entry of row $i$. Hence $A$ is diagonally dominant and
therefore non-singular by the Levy–-Desplanques theorem; see for example \cite{horn}.
\end{proof}

The next lemma is an interesting result on adjacent Steiner points in locally minimal FQSTs. The lemma is also useful for proofs in Sections
\ref{dBound} and \ref{secTop}. Let $s_0,s_1$ be any two adjacent Steiner points of a locally minimal FQST $T$, where $s_1,z_1,...,z_j$ are the
neighbours of $s_0$; $s_0,z_{j+1},...,z_{p+1}$ are the neighbours of $s_1$; and $z_{p+1}$ is the local sink. Once again let $f_i$ be the flow
associated with the edge incident to $z_i$. Let $I=\{1,...,p+1\}$, $J=\{1,...,j\}$ for $j\leq p$, and $\overline{J}=I-J$. Let
$C_J=\mathcal{C}(\{f_iz_i\}_{i\in J})$ and $F_J=\displaystyle\sum_{i\in J}f_i$, and similarly for $C_{\overline{J}}$ and $F_{\overline{J}}$.

\begin{figure}[htb]
\begin{center}
    \includegraphics[scale=0.5]{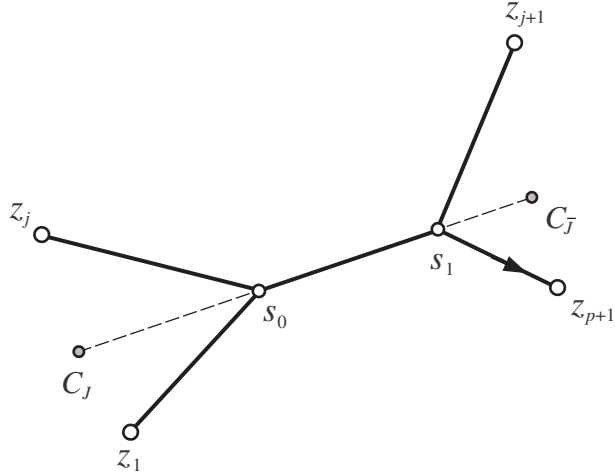}\\

  \caption{The points $C_J,s_0,s_1,C_{\overline{J}}$ are collinear}
  \label{figColl}
\end{center}
\end{figure}

\begin{lemma}With the above notation and definitions, the points $C_J,s_0,s_1,C_{\overline{J}}$ are collinear and subdivide the line segment
$C_JC_{\overline{J}}$ into a ratio $F_{\overline{J}}:F_{\overline{J}}:F_J$.
\end{lemma}
\begin{proof}The lemma is illustrated in Fig.~\ref{figColl}. Note that $s_1=\mathcal{C}\left(F_Js_0,F_{\overline{J}}C_{\overline{J}}\right)$ and
$s_0=\mathcal{C}(F_JC_J,F_{\overline{J}}s_1)$, from which the result follows.
\end{proof}

\subsection{Properties of degree-bounded MFQSTs}\label{dBound}
We next examine some properties of degree bounded MFQSTs, where the given lower bound on degree is $\phi\geq 3$. Let $x\in Z\cup S$ be a node in
a locally minimal FQST $T$, with in-neighbours $z_1,...,z_p$ and local sink $z_{p+1}$. Let $I=\{1,...,p+1\}$ and let $J\subseteq
I\backslash\{p+1\}$. A \textit{$J$-split} (or just \textit{split} if the context is clear) of $x$ introduces a Steiner point $s'$ such that the
neighbours of $s'$ are $x$ and every $z_i$, $i\in J$, and the neighbours of $x$ are $s'$ and $z_i$ for every $i\in I\backslash J$. We assume
that, after splitting, $s'$ and all other Steiner points are relocated to their optimal positions relative to the new topology. Let the
resultant tree be denoted by $T_0$. A $J$-split is \textit{beneficial} if $L(T_0)<L(T)$. Most splits are beneficial, but not all; see Fig.
\ref{figNoSplit}.

In this section we use the existence of beneficial $J$-splits to show that imposing a lower bound $\phi \geq 3$ on the degree of Steiner points
implies there is also an upper bound of $2\phi-3$ on the degree. We need the following definition and lemma before we can show that beneficial
splits of any size can always be found.

We define two edges in a network to be \emph{overlapping} if they are both incident with a common node and every point of one edge lies in the
other.

\begin{figure}[htb]
\begin{center}
    \includegraphics[scale=0.4]{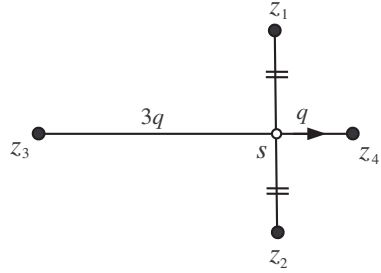}\\

  \caption{A $\{1,2\}$-split of $s$ is not beneficial since $\mathcal{C}(1z_1,1z_2)=\mathcal{C}(1z_3,3z_4)=s$}
  \label{figNoSplit}
\end{center}
\end{figure}

\begin{lemma}\label{lemsplit} Let $T$ be a  non-degenerate FQST  with a pair of  overlapping edges, where either $T$ is not degree-bounded,
or $T$ is degree bounded and the common node of the two overlapping edges is not a Steiner point of degree $\phi$. Then $T$ is not an MFQST.
\end{lemma}
\begin{proof}
Suppose that $T$ is as specified but is also an MFQST. Suppose that the two overlapping edges are $e_1=u_1v$ and $e_2=u_2v$ with
$|u_1v|>|u_2v|$. Then, regardless of the directions of flow, we can replace edge $e_1$ by $u_1u_2$ and thereby reduce the total cost of $T$.
This is a contradiction. Note that the edge replacement is allowed in the degree-bounded case because the only  node-degree that decreases is
that of $v$, which, by assumption, is not a Steiner point of degree $\phi$.
\end{proof}

In the next proposition we let $d_s$ denote the in-degree of $s$.

\begin{proposition}\label{propsplit} Let $T$ be a non-degenerate locally minimal FQST  containing a Steiner point $s$ such that $d_s>\phi-1$ if $T$ is
degree-bounded (with degree bound $\phi$), and  $d_s>1$ otherwise. Furthermore, assume that  no pair of overlapping edges in $T$ have  $s$ as
their common node. Then there exists a beneficial $J$-split of $s$ for any $1\leq |J|\leq d_s$.
\end{proposition}
\begin{proof}Suppose that the in-neighbours of $s$ are $z_1,...,z_{d_s}$ and let $J\subseteq \{1,...,d_s\}$. We argue by induction on the cardinality
of $J$. Clearly the result holds for $|J|=1$ since $\mathcal{C}(f_iz_i)=z_i\neq s$ for any $i\leq d_s$ by non-degeneracy; also, the result holds
for $|J|=d_s$ since $\mathcal{C}(f_1z_1,...,f_{d_s}z_{d_s})\neq s$ by Corollary~\ref{corMid}. Next suppose that $|J|=d$ for some $1\leq d\leq
d_s-2$ such that the $J$-split of $s$ is beneficial, and let $\alpha,\alpha^\prime\in  \{1,...,d_s\}\backslash J$ be distinct. Let
$J_0=J\cup\{\alpha\}$ and let $J_1=J\cup\{\alpha'\}$. Note that if $C_{J_0}=C_{J_1}=s$ then the three points $z_\alpha,z_{\alpha'}$ and $s$ are
collinear, with $z_\alpha$ or $z_{\alpha'}$ lying between the other two (again by Corollary~\ref{corMid}).  This contradicts the  assumption
that $s$ has no incident pair of overlapping edges. Therefore, either the $J_0$-split or $J_1$-split of $s$ is beneficial and is of cardinality
$d+1$.
\end{proof}

We now have the following four results for the degree bounded problem. The two corollaries also hold for Euclidean quadratic Steiner trees; see
\cite{soukop}.

\begin{proposition}An MFQST $T$ with degree bound $\phi$ has $\phi\leq \deg s \leq 2\phi-3$ for every Steiner point $s$.
\end{proposition}
\begin{proof}Suppose, contrary to the proposition, that $T$ contains a Steiner point $s$ with $\deg s > 2\phi-3$. Note that $s$ has
no incident pair of overlapping edges, by Lemma~\ref{lemsplit}. Any $J$-split of $s$ with $|J|=\phi-1$ will produce two Steiner points each of
degree at least $\phi$. By  Proposition~\ref{propsplit} at least one choice of such $J$ must be beneficial, which contradicts the minimality of
$T$.
\end{proof}

\begin{corollary}When $\phi=3$ every Steiner point will be of degree exactly $3$.
\end{corollary}

\begin{proposition}Every source $z$ in an MFQST $T$ with degree bound $\phi$ is of degree at most $\phi-1$. Moreover, if  $z$ has degree equal to
$\phi-1$ and if $C$ denotes the centre of mass of $z$ and its in-neighbours, then $z$ lies at the midpoint of $C$ and the out-neighbour of $z$.
\end{proposition}

\begin{proof}The reasoning is similar to the previous proposition, except that when using $J$-splits there is no lower bound on the in-degree of $z$.
\end{proof}

\begin{corollary}Every source in an MFQST with $\phi=3$ is of degree at most two. Moreover, if a source has degree equal to $2$ then it is
collinear with its neighbours (see for instance Fig. \ref{fig6}).
\end{corollary}

\begin{figure}[htb]
\begin{center}
    \includegraphics[scale=0.4]{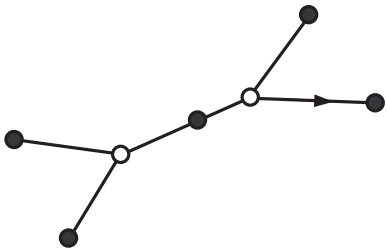}\\

  \caption{An MFQST  with $\phi=3$ with a source of degree $2$.}
  \label{fig6}
\end{center}
\end{figure}

\subsection{Properties of MFQSTs that are not degree-bounded}
In this subsection we consider some properties of the node-weighted and the explicitly bounded versions of the FQSTP.

\begin{figure}[htb]
\begin{center}
    \includegraphics[scale=0.4]{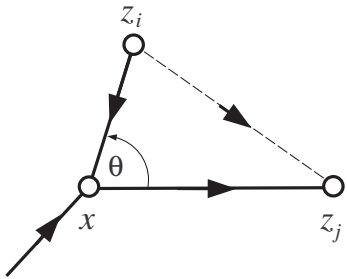}\\
  \caption{The angle between  an in-edge and an out-edge must be obtuse}
  \label{fig1}
\end{center}
\end{figure}

\begin{proposition}\label{Propangle} The angle between any in-edge and the out-edge of a given node in a node-weighted or explicitly bounded MFQST is
at least $90^\circ$ (see Fig. \ref{fig1}).
\end{proposition}
\begin{proof}If we assume that $\theta < 90^\circ$ as in Fig. \ref{fig1} then path $z_i,z_j$ is shorter (quadratically) than path $z_i,\mathbf{x},z_j$
(by Pythagoras' theorem). Therefore, since there are no restrictions on the degree of nodes, we can replace edge $z_i\mathbf{x}$ with edge
$z_iz_j$, resulting in a tree with lower cost.
\end{proof}

In fact, by repeatedly swapping edges we can ensure that every in-edge-out-edge angle is strictly greater than $90^\circ$, although this may
cause the sink to obtain a large degree

\begin{figure}[htb]
\begin{center}
    \includegraphics[scale=0.4]{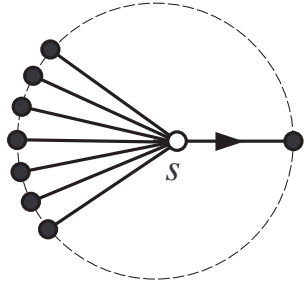}\\
  \caption{Nodes in an MFQST can acquire large degree}
  \label{fig2}
\end{center}
\end{figure}

\begin{proposition}Steiner points of a node-weighted or explicitly bounded MFQST can achieve a degree of $|Z|+1$.
\end{proposition}
\begin{proof}For the explicitly bounded case, suppose that sources are located on a circle as in Fig. \ref{fig2}, with a single Steiner point. For
this FQST to be locally minimal, the Steiner point must be located near the centre of the circle, but slightly towards  the sink. Clearly, by
Proposition~\ref{Propangle}, this tree is also globally minimal -- in particular, if any of the sources has degree greater than 1 in the MFQST
we get a contradiction to the angle condition. Therefore the Steiner point for the  MFQST with explicit bound $k=1$ in this instance has degree
$|Z|+1$.

For the node-weighted case, let $|Z|=n$, and suppose the sources and sink are located on a unit diameter circle such that all sources lie within
an arbitrarily small neighbourhood of the antipodal point to the sink. Then for any number of Steiner points the minimum cost of the network
will be close to that for a network with a single source of weight $n$. Let the cost of each Steiner point be $n/2 - \varepsilon$ for a small
$\varepsilon >0$. Then it is clear that the MFQST has exactly one Steiner point, located close to the centre of the unit diameter circle. By the
same argument as in the previous case, this Steiner point has degree $n+1$.
\end{proof}

Next we present a few properties relating to the node-weighted version of the FQSTP. Recall that this version utilises a modified cost function,
namely $L_c(T)=\displaystyle\sum_{e\in E(T)} f(e)\vert e \vert^2+c|S|$ where $c>0$ is the cost of a Steiner point. The first two results are
necessary conditions for an FQST to be an MFQST, and are based on the number of degree two Steiner points that lie on any given straight line
path of an optimal tree.

\begin{proposition}Let $p$ be the number of degree-two Steiner points located on a path $P$ with endpoints $u,v$ of a node-weighted MQFST. Then
$p(p+1)\leq\displaystyle\frac{f(uv)|uv|^2}{c}\leq (p+2)(p+1)$, where $c$ is the cost of a Steiner point.
\end{proposition}
\begin{proof}Note that the nodes on $P$ are collinear and equally spaced along the segment $uv$. The cost of path $P$ is
$\displaystyle\frac{f(uv)|uv|^2}{p+1}+cp$. We therefore need
$\displaystyle\frac{f(uv)|uv|^2}{p+1}+cp\leq\displaystyle\frac{f(uv)|uv|^2}{p+2}+c(p+1)$ and
$\displaystyle\frac{f(uv)|uv|^2}{p+1}+cp\leq\displaystyle\frac{f(uv)|uv|^2}{p}+c(p-1)$, from which the result follows from simple algebra.
\end{proof}

\begin{corollary}In a node-weighted MFQST every edge $e$ satisfies $|e|\leq \sqrt{\displaystyle\frac{2c}{f_e}}$, where $c$ is the cost of a Steiner
point.
\end{corollary}

Since we are not directly bounding $k$ in the node-weighted version, it would be helpful to determine an upper bound $B$ for $k$ in terms of
$c$. This would immediately lead to an exact algorithm for calculating node-weighted MFQSTs: the first step would be to iterate through all
topologies interconnecting the given sources and at most $B$ Steiner points. A locally optimal solution for every topology could then be
calculated using either the algebraic or geometric methods described the next section, and the cheapest tree selected. The complexity of this
method, however, would be prohibitive for large problems (i.e., large $n$ or small $c$) and effective methods of pruning the number of viable
topologies would be needed in these cases. We first prove the following result in order to bound the ``length component" of the cost of a
node-weighted MFQST.

\begin{lemma}\label{lemB}Let $T$ be an explicitly bounded MFQST on $n$ sources, with bound $k$. Then
$L(T)\geq (n+k+1)^{-1}\displaystyle\sum_{i\leq n} |z_iz_{\mathrm{BS}}|^2$.
\end{lemma}
\begin{proof}Consider a single unit of flow from source $z_i$. A cheapest path from $z_i$ to $z_{\mathrm{BS}}$ must have a cost of at least that of
the path $P$, where $P$ contains all Steiner points and remaining sources arranged collinearly and equally spaced between $z_i$ and
$z_{\mathrm{BS}}$. In this case the cost of $P$ would be $\displaystyle\frac{|z_iz_{\mathrm{BS}}|^2}{n+k+1}$.
\end{proof}

Let $T$ be a minimum spanning tree on $Z\cup\{z_{\mathrm{BS}}\}$, and suppose that we add the minimum number of degree-two Steiner points to $T$
such that no edge is longer than $\sqrt{\frac{2c}{f}}$, where $f$ is the flow on the edge. Note that this can be done greedily, and therefore
the construction is of polynomial complexity. Let the resultant tree be denoted by $BST(Z)$, i.e., the \textit{beaded spanning tree} on $Z$.

\begin{lemma}Let $Z\cup\{z_{\mathrm{BS}}\}$ be any set of sources and a sink, and let $T_{\mathrm{opt}}$ be a node-weighted MFQST on these nodes. Then
$L_c(T_{\mathrm{opt}})\leq L_c(BST(Z))$.
\end{lemma}

In the previous lemma it may be possible to construct a spanning tree (in polynomial time) that provides a tighter bound than a minimum spanning
tree. An improvement of this bound, or the bound in Lemma \ref{lemB}, would most likely lead to a better value of $B$ in the next proposition.

\begin{proposition}Suppose that a node-weighted MFQST $T$ contains $k$ Steiner points. Then
$k\leq\displaystyle\frac{1}{c}\left\{L_c(BST(Z))-(n+k+1)^{-1}\displaystyle\sum |z_iz_{\mathrm{BS}}|^2\right\}$, which can be rewritten (by
solving a quadratic equation in $k$, in order to make $k$ the subject) in the form $k\leq B$ where $B$ does not contain $k$.
\end{proposition}
\begin{proof}The result follows directly from the previous two lemmas after noting that $k=\displaystyle\frac{1}{c}\left(L_c(T)-L(T)\right)$.
\end{proof}

\section{A Geometric Linear-time algorithm for fixed topologies}\label{secTop}
In this section we describe a geometric linear time algorithm for constructing an MFQST with a given full topology, but where all Steiner points
are assumed to have degree 3. This is a key step in any  general algorithm for constructing MFQSTs over all possible topologies. It can
potentially be combined with an exhaustive search, along with appropriate pruning methods, to build an exact method for finding MFQSTs (along
the lines of the method underlying GEOSTEINER \cite{warme}), or can be combined with appropriate heuristic search techniques as part of an
approximation algorithm. As mentioned before, for many variants of the problem the restriction of the degree of Steiner points (especially to
degree 3) is a very natural one. An algebraic linear time algorithm does exist in the form of a solution to the system of diagonally dominant
linear equations discussed in the proof of Theorem \ref{thrIff} (see also \cite{ganley} for the related algorithm without flow), but this
algorithm reveals very little directly about the structure of locally optimal FQSTs.

The general strategy of the algorithm is similar to Melzak's algorithm for constructing fixed topology Euclidean Steiner trees; see for instance
\cite{hwang}. We begin by finding two sources that are adjacent to a single Steiner node in the given topology, and replace the pair of sources
by a single \textit{quasi-source} whose location and mass can be explicitly computed. This procedure is repeated recursively: at each stage
there exists a Steiner point adjacent to two nodes, each of which is either a source or quasi-source; these nodes along with the Steiner point
are replaced by a new quasi-source. The procedure continues until there are no sources left, and only a single quasi-source in the tree. The
position and mass of this quasi-source allows one to construct the Steiner point adjacent to the sink, and then a backtracking procedure allows
one to construct each of the remaining Steiner points in turn.

\textbf{Notation:} As before, let $Z=\{z_1,\ldots , z_n\}$ be the set of sources for a full MFQST $T$, and let $z_{\rm {BS}}$ be the sink. We
think of each source and the sink as being a vector in $\mathbb{R}^2$ representing the position of the node in Cartesian coordinates. Associated
with each source $z_i$ is a mass $w(z_i)$, representing the amount of flow from that source. As in the earlier sections, we assume  that
$w(z_i)=1\ \forall i \in \{1,\ldots , n\}$ under a suitable choice of units, however the algorithm can easily be adapted to situations where
different sources have different flows. Let ${S}=\{s_1,\ldots , s_{n-1}\}$ be the set of Steiner points of $T$. We associate a mass with each
Steiner point $s_i$ and the sink additively;  for example, the mass of each $s_i$ is the sum of the masses of the two nodes whose out-edges are
the in-edges of $s_i$. Nodes that are not Steiner points (including the sink) will be referred to as \textit{terminals}.

We now define the concept of a \emph{quasi-source}. Given an MFQST containing a Steiner point adjacent to two terminals (neither of which is a
sink), a quasi-source is a new terminal that replaces these three nodes, such that the remaining Steiner points of the resulting MFQST on this
reduced set of terminals are in the same locations as in the original tree. More formally, let $T$ be a full MFQST  with terminals $Z \cup
\{z_{\rm {BS}}\}$ (where $z_{\rm {BS}}$ is the sink), Steiner points $S = \{ s_1,\ldots , s_{n-1}\}$ each with a given mass, and full topology
$\mathcal{T}$. Suppose we are given $z_1, z_2 \in Z$ and $s_1 \in S$  such that $z_1$ and $z_2$ are each adjacent to $s_1$, and let $v$ be the
third node  (other than $z_1$ and $z_2$) adjacent to $s_1$. This is illustrated in Fig. \ref{trans}.
\begin{figure}[ht]
\begin{center}
 \includegraphics[width=4in]{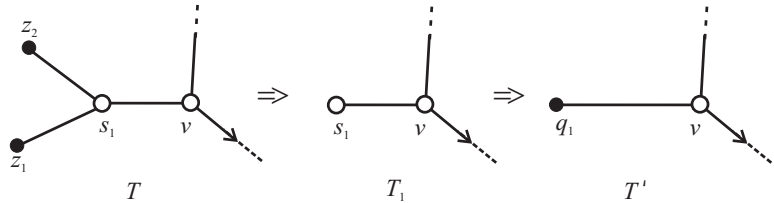}
 \caption{The transition from $T$ to $T'$ (via $T_1$).}\label{trans}
\end{center}
\end{figure}
Let $T_1$ be the subtree of $T$ in which $z_1$, $z_2$ and their two incident edges have been removed, and let $\mathcal{T}_1$ be the topology of
$T_1$. Note that if we treat $s_1$ as a terminal, then $\mathcal{T}_1$ is a full topology. Given a point $q_1$ with mass $w(q_1)$ let $T'$ be
the MFQST with terminals  $\{q_1,z_3,\ldots , z_n\} \cup z_{\rm {BS}}$, topology $\mathcal{T}_1$ (where node $s_1$ has been relabelled $q_1$),
and where the weights of the Steiner points in $T'$ are the same as their weights in $T$. Then $q_1$ is said to be a  \emph{quasi-source}
replacing $z_1, z_2$ and $s_1$, if the Steiner points of $T'$ are in the same locations as the corresponding Steiner points of $T$ and $s_1=
\mathcal{C}(w(q_1) q_1, w(v) v)$. The transition from $T$ to $T'$ (via $T_1$) is illustrated in Fig.~\ref{trans}.

It is important to note that the above definition involves a slight abuse of notation -- strictly speaking, $T'$ is not an MFQST since the mass
assigned to a quasi-source no longer has an obvious interpretation in terms of flow. It is convenient, however,  to treat $T'$ as though it is
an MFQST. Because the weights of Steiner points do not change when we introduce a quasi-source, it follows that such `MFQSTs' that include
quasi-sources as terminals do not necessarily have the additive property for all Steiner points (but, by definition, still satisfy the centre of
mass properties).

In the lemmas that follow we show that for a given $T$ we can always find a quasi-source $q$, such that the position and mass of $q$ depend only
on the known masses of nodes in the tree and the locations of the two terminals it is replacing. We distinguish three methods of constructing
such a quasi-source, depending on the nature of the two terminals it replaces.

\begin{lemma}\label{qs1}
Let $z_1$ and $z_2$ be two sources of $T$, both adjacent to a single Steiner point $s$. Let $q$ be the midpoint of $z_1$ and $z_2$ (ie, $q= (z_1
+ z_2)/2$) and let $w(q)=2$. Then $q$ is a quasi-source replacing $z_1$ and $z_2$.
\end{lemma}

\begin{proof}
Note that $q$ is the centre of mass of $z_1$ and $z_2$, and $w(q)=w(z_1) + w(z_2)$. Let $v$ be the third neighbour of $s$, other than $z_1$ and
$z_2$. It follows that the centre of mass of $z_1$, $z_2$ and $v$ is the same as the centre of mass of $q$ and $v$. Hence the  result follows.
\end{proof}

\begin{lemma}\label{qs2}
Let $q$ and $z$ be two terminals of $T$, a quasi-source and source respectively, both adjacent to a single Steiner point $s_2$. Let $s_1$ be the
Steiner point (not in $T$) adjacent to the two terminals replaced by $q$ in a previous minimum quadratic flow-dependent Steiner tree $T_1$.  Let
$w_0= w(q)$ and $w_1=w(s_1)$ (in $T_1$). Define the point $q_1$ as follows:
$$q_1:= q + \frac{w_0+w_1}{w_0+w_1 + w_0w_1} (z-q)$$
with mass
$$w(q_1):=\frac{w_0+w_1 + w_0w_1}{w_0+w_1}. $$
Then $q_1$ is a quasi-source  replacing $q$, $z$ and $s_2$.
\end{lemma}

\begin{proof}
The aim is to show that the point $q_1$ as defined in the statement of the lemma is a suitable quasi-source for replacing $q$, $z$ and $s_2$.
Consider the construction illustrated in Fig.~\ref{lemqz}.
\begin{figure}[ht]
\begin{center}
 \includegraphics[width=3.5in]{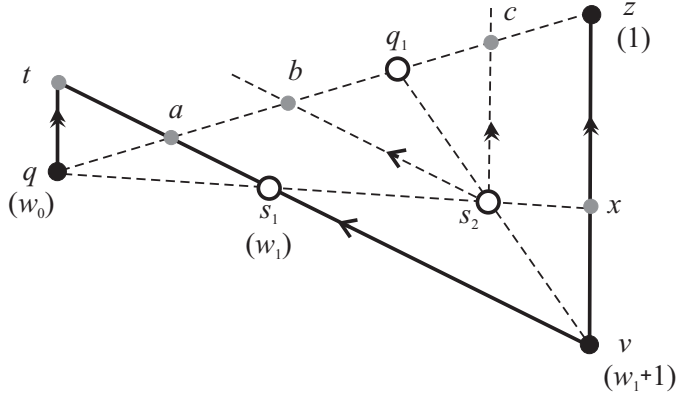}
 \caption{Construction for the proof of Lemma \ref{qs2}. The weights of some points are shown in parentheses. The single and double
arrows are used to indicate parallel lines.}\label{lemqz}
\end{center}
\end{figure}
Let $v$ be the third node adjacent to $s_2$ other than $s_1$ and $z$. Note that $v$ is either a Steiner point or the sink, and $w(v) = w_1+1$.

In $T_1$, $s_1$ is the centre of mass of $q$ and $s_2$. Hence $q$, $s_1$ and $s_2$ are collinear. Let $x$ be the point where the line through
$q$, $s_1$ and $s_0$ intersects the line segment $vz$. By mass-point geometry,
\begin{equation}\label{qse1}
 \frac{|vx|}{|xz|} = \frac{1}{w_1+1}.
\end{equation}
Furthermore, we can consider $x$ to have an associated mass of $w(x) = w_1 +2$. It follows that $\frac{|qs_1|}{|s_1s_2|} = \frac{w_1}{w_0}$ and
$\frac{|s_1s_2|}{|s_2x|} = \frac{w_1+2}{w_1}$, which, in terms of ratios, gives:
\begin{equation}\label{qse2}
 |qs_1|:|s_1s_2|:|s_2x|  = w_1(w_1+2): w_0 (w_1+2): w_1 w_0 .
\end{equation}
Now, let $L$ be the line through $q$ parallel to $vz$, and let $t$ be the point where the line through $s_1$ and $v$ intersects $L$. Let $a$ be
the point where $s_1t$ intersects $qz$. Since $\triangle tqs_1 \sim \triangle vxs_1$ it follows from (\ref{qse2}) that $\frac{|tq|}{|xv|} =
\frac{w_1(w_1+2)}{2w_0(w_1+1)}$, and hence, by (\ref{qse1}), $\frac{|tq|}{|zv|} = \frac{w_1}{2w_0(w_1+1)}$. Since, $\triangle tqa \sim \triangle
vza$, we now obtain:
\begin{equation}\label{qse3}
 \frac{|qa|}{|qz|} = \frac{w_1}{w_1+2w_0(w_1+1)}.
\end{equation}
Let $b$ be the point on $qz$ such that $s_2 b \parallel s_1 a$. Since $\triangle qas_1 \sim \triangle qbs_2$, we obtain from (\ref{qse2}) and
(\ref{qse3}):
\begin{equation}\label{qse4}
 \frac{|ab|}{|qz|} = \frac{w_0}{w_1+2w_0(w_1+1)}.
\end{equation}
Similarly, let $c$ be the point on $qz$ such that $s_2 c \parallel xz$. Since $\triangle cs_2q \sim \triangle zxq$, we deduce that:
\begin{equation}\label{qse5}
 \frac{|cz|}{|qz|} = \frac{w_1w_0}{(w_1+w_0)(w_1+2)+ w_1w_0}.
\end{equation}

Equations~(\ref{qse3}), (\ref{qse4}) and (\ref{qse5}), give us the locations of points $a$, $b$ and $c$ respectively. Now let $q_1$ be the
intersection of the line through $v$ and $s_2$ with $qz$. Since $\triangle bcs_2 \sim \triangle azv$, we obtain:
\begin{equation}\label{qse6}
 \frac{|bq_1|}{|bc|} = \frac{|ab|+|bq_1|}{|az|}.
\end{equation}
Equation~(\ref{qse6}) allows us to compute $|bq_1|$, resulting in the location of $q_1$ as given in the statement of the lemma. Finally, we note
that $\triangle bq_1s_2 \sim \triangle aq_1v$, hence $\frac{|q_1s_2|}{|q_1v|} = \frac{|bq_1|}{|aq_1|}$ from which we obtain the mass of $q_1$ as
stated in the lemma.
\end{proof}

\begin{lemma}\label{qs3}
Let $q_1$ and $q_2$ be two quasi-sources of $T$, both adjacent to a single Steiner point $s_3$. For $i=1,2$, let $s_i$ be the Steiner point
adjacent to the two terminals replaced by $q_i$ in a previous minimum quadratic flow-dependent Steiner tree $T_i$, and let $w_{0i}= w(q_i)$ and
$w_i=w(s_i)$ (in $T_i$). Define the point $q_3$ as follows:
$$q_3:= q_1 + \frac{w_{02}w_2(w_1+w_2)}{w_1w_2(w_{01} + w_{02}) + w_{01}w_{02}(w_{1} + w_{2})} (q_2-q_1)$$
with mass
$$w(q_3):=\frac{w_1w_2(w_{01} + w_{02}) + w_{01}w_{02}(w_{1} + w_{2})}{(w_{1} + w_{01})(w_{2} + w_{02})}. $$
Then $q_3$ is a quasi-source for $T$, replacing $q_1$, $q_2$ and $s_3$.
\end{lemma}

\begin{proof}
The proof is similar to the proof of Lemma~\ref{qs2}. Consider the construction given in Fig.~\ref{lemqq}, where again $v$ is the third node
adjacent to $s_3$ other than $s_1$ and $s_2$, and $w(v) = w_1 + w_2$.
\begin{figure}[ht]
\begin{center}
 \includegraphics[width=3.5in]{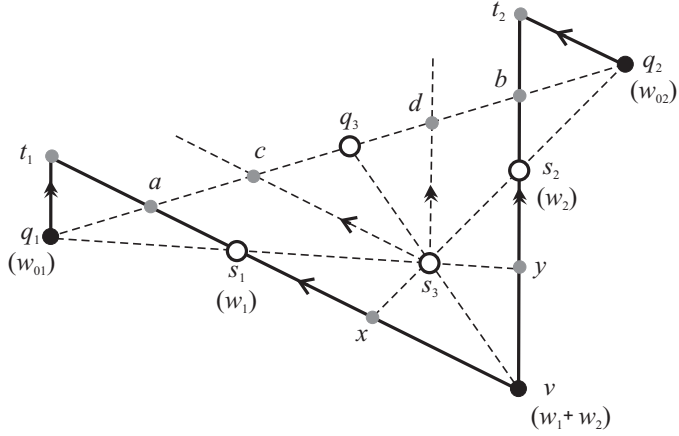}
 \caption{Construction for the proof of Lemma~\ref{qs3}.}\label{lemqq}
\end{center}
\end{figure}
 let $L_1$ be the line through $q_1$ parallel to $vs_2$, and let $L_2$ be the line
through $q_2$ parallel to $vs_1$; for $i=1,2$, let $t_i$ be the point where the line through $s_i$ and $v$ intersects $L_i$. As in the proof of
Lemma~\ref{qs2}, we can compute the points $a$ and $b$, where $q_1q_2$ intersects $s_1t_1$ and $s_2t_2$ respectively, and the points $c$ and $d$
on $q_1q_2$ such that $s_3c$ and $s_3 d$ are parallel to $va$ and $vb$ respectively. This again allows us to locate the point $q_3$  at the
intersection of the line through $v$ and $s_3$ with $q_1q_2$, with location and mass as given in the statement of the lemma.
\end{proof}

We now describe the algorithm for locating the Steiner points of $T$ by successively replacing pairs of terminals by quasi-sources using the
above three lemmas.

\begin{quote}
\noindent \textbf{Algorithm MFQST}

\smallskip

    \textbf{Input:} A set $Z$ of $n$ sources, a sink $z_{\rm {BS}}$, and
     a full topology $\mathcal{T}$ for $Z \cup \{z_{\rm {BS}} \}$.

    \smallskip

    \textbf{Output:} A minimum flow-dependent quadratic Steiner tree $T$  for $Z \cup \{z_{\rm {BS}} \}$ with topology $\mathcal{T}$, along with its
    cost $L(T)$.

    \begin{enumerate}
      \item For each $z_i \in Z$ set $w(z_i)=1$; for each Steiner point $s_j$ of $\mathcal{T}$ compute $w(s_j)$, via additivity; set
      $w(z_{\rm {BS}})= n$.
      \item Find a Steiner point adjacent to two terminals (each of which is a source or quasi-source) and replace the three nodes by a new
      quasi-source, using Lemmas~\ref{qs1}, \ref{qs2} and \ref{qs3}. Update the Steiner tree. Repeat until all Steiner points have been replaced.
      \item The Steiner tree will now contain only two terminals, a quasi-source and the sink. Use recursive back-tracking to determine the positions
      of the Steiner points (in the reverse order to the order of replacement in the previous stage) where each Steiner point is at the centre of mass
      of its two neighbouring nodes.
      \item Once all quasi-source replacements have been undone, the tree $T$ with all of its Steiner points will have been correctly constructed.
      The cost of $T$ is computed using node weights and edge lengths.
    \end{enumerate}
\end{quote}

Before proving the correctness of Algorithm MFQST, it is helpful to illustrate the running of the algorithm with a simple example, shown in Fig.
\ref{example}, where all locations are given in Cartesian coordinates.
\begin{figure}[ht]
\begin{center}
 \includegraphics[width=4in]{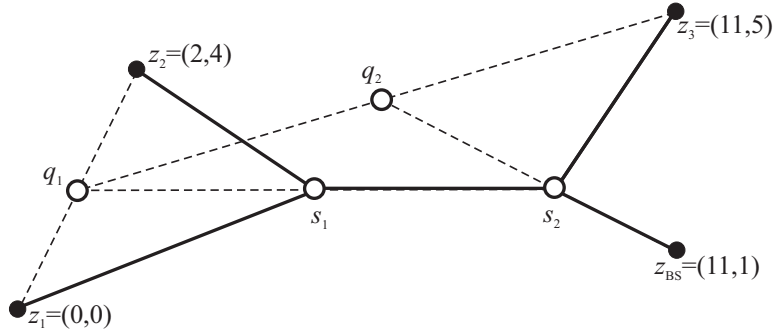}
 \caption{Example of the construction of an MFQST using the algorithm.}\label{example}
\end{center}
\end{figure}
The example contains four terminals: three sources $z_1= (0,0), z_2=(2,4), z_3 = (11,5)$, and a sink $z_{\rm BS} = (11,1)$. The full topology of
the tree is shown in unbroken lines. In Step~1, the weights of the two Steiner points are computed by additivity: $w(s_1)=2$ and $w(s_2)=3$. In
Step~2 we replace each Steiner point and two adjacent terminals by a quasi-source. The first quasi-source $q_1$ replaces $s_1$ and the two
sources $z_1$ and $z_2$. Hence, by Lemma~\ref{qs1}, $q_1=(z_1+z_2)/2=(1,2)$ and $w(q_1)=2$. In the resulting tree, the remaining Steiner point
$s_2$ is adjacent to $q_1$ and $z_3$. We replace these three nodes by a new quasi-source $q_2$ where, by Lemma~\ref{qs2}, $q_2= q_1+(z_3-q_1)/2=
(6, 3.5)$ and $w(q_2)=2$. This concludes Step~2. For Step~3 we determine the positions of the Steiner points in reverse order to Step~2. The
Steiner point $s_2$ lies at the centre of mass of $q_2$ and $z_{\rm BS}$, hence $s_2=\frac{2}{5}q_2 + \frac{3}{5}z_{\rm BS}=(9,2)$. Similarly,
$s_1$ lies at the centre of mass of $q_1$ and $s_2$, where, for this branch the relevant mass of $s_2$ corresponds to the flow on the edge
$s_1s_2$ in the final tree (ie, $2$). Hence $s_1=\frac{1}{2}q_1 + \frac{1}{2}s_2=(5,2)$. This completes Step~3, giving the MFQST $T$ with
$L(T)=102$.

\begin{theorem}\label{algorithm}
Given a set of terminals $Z$ and a corresponding full topology $\mathcal{T}$, Algorithm MFQST correctly computes a MFQST $T$ for $Z$.
Furthermore, the algorithm runs in time $O(n)$.
\end{theorem}

\begin{proof}
We claim that, in Step~2 of the algorithm, if the current tree contains Steiner points then there is at least one Steiner point adjacent to two
terminals, neither of which is the sink. If the tree contains only one Steiner point then the statement is trivial, while if the tree contains
more than one Steiner point, then this follows from the observation  that the subtree induced by the Steiner points contains at least two
vertices of degree~1. Hence we know that such a Steiner point exists at each stage, and the correctness of the algorithm easily follows from
Lemmas~\ref{qs1}, \ref{qs2} and \ref{qs3}.

For the running time, note that the order of replacing Steiner points by quasi-sources can be pre-determined by running a depth-first search on
the topology, which can be done in linear time. The construction of each quasi-source can be done in constant time, using the formulas in
Lemmas~\ref{qs1}, \ref{qs2} and \ref{qs3}, hence Step~2 of the algorithm runs in linear time. Similarly, Step~3, determining the position of
each Steiner point, requires only linear time.
\end{proof}

Finally we note that the methods in this section can easily be extended to allow Steiner points of both degree 2 and 3. The inclusion of higher
degree Steiner points, however, appears to significantly increase the complexity of the problem, and may require a different approach.

\section{Conclusion}
In this paper we introduced a flow-dependent version of the quadratic Steiner tree problem in order to model optimal relay deployment in
wireless networks, specifically networks that transmit in free space or in situations where constructive interference applies. We described some
structural geometric properties of locally minimal solutions to the problem, including properties relating to the degrees and locations of
Steiner points. We did this under various strategies for bounding the number of Steiner points.  Finally, we described a new geometric algorithm
for constructing locally minimal solutions. The algorithm is based on the mass-merging method of mass-point geometry, and runs in linear time,
matching the fastest known algebraic algorithm for the problem.


\end{document}